\documentclass[12pt,reqno]{amsart}

\usepackage{fullpage,xcolor}

\newtheorem{theorem}{Theorem}
\newtheorem{lemma}{Lemma}
\newtheorem{corollary}[lemma]{Corollary}

\newcommand{\eps}{\varepsilon}
\newcommand{\dtheta}[1][\tau]{\mathrm d_{#1}\theta}

\title{A Hybrid of Two Theorems of Piatetski-Shapiro}
\author{Angel Kumchev}
\address{Department of Mathematics\\ Towson University\\ Towson, MD 21252\\ U.S.A.}
\email{akumchev@towson.edu}

\author{Zhivko Petrov}
\address{Faculty of Mathematics and Informatics\\ Sofia University\\ Sofia\\5 J. Bourchier, 1164 Sofia,  Bulgaria}
\email{zhpetrov@fmi.uni-sofia.bg}

\date{}

\begin{document}

\maketitle

\begin{abstract}
Let $c > 1$ and $0 < \gamma < 1$ be real, with $c \notin \mathbb N$. We study the solubility of the Diophantine inequality
\[ \left| p_1^c + p_2^c + \dots + p_s^c - N \right| < \varepsilon \]
in Piatetski-Shapiro primes $p_1, p_2, \dots, p_s$ of index $\gamma$---that is, primes of the form $p = \lfloor m^{1/\gamma} \rfloor$ for some $m \in \mathbb N$.
\end{abstract}

\section{Introduction}

Let $0<\gamma<1$ be a fixed real number. I.I. Piatetski-Shapiro \cite{PS53} was the first to consider the question whether the sequence
\[ \mathcal N_\gamma = \big\{ n \in \mathbb N : n = \big\lfloor m^{1/\gamma} \big\rfloor \text{ for some } m \in \mathbb N \big\} \]
contains infinitely many primes. He proved that when $\gamma > 11/12$, one has the asymptotic formula
\begin{equation}\label{eq:i.1}
\pi_\gamma(N) = \frac {N^\gamma}{\log N} \big( 1 + O\big( (\log N)^{-1} \big) \big)
\end{equation}
for the number $\pi_\gamma(N)$ of primes $p \le N$ that belong to $\mathcal N_\gamma$. This result has attracted a lot of attention, and a number of authors \cite{BHR95,HB83,Ji93,Ji94,Ko67,Ko85,Ku99b,LR92,RS01,RW01} have extended the range of $\gamma$ for which $\pi_\gamma(N) \to \infty$ to $\gamma > 205/243 = 0.8436\dots$ (see Rivat and Wu~\cite{RW01}). In the process, prime numbers $p \in \mathcal N_\gamma$ have become known as \emph{Piatetski-Shapiro primes} (of index $\gamma$).

Another problem proposed by Piatetski-Shapiro \cite{PS52} around the time he proved \eqref{eq:i.1} deals with the solubility in primes $p_1, p_2, \dots, p_s$ of the Diophantine inequality
\begin{equation}\label{eq:i.2}
\big| p_1^c + p_2^c + \dots + p_s^c - N \big| < \eps,
\end{equation}
where the exponent $c > 1$ is not an integer, $\eps > 0$ is a fixed small number, and $N$ is a large real. If $H(c)$ denotes the least integer $s$ such that \eqref{eq:i.2} has solutions for all sufficiently large $N$, then Piatetski-Shapiro proved that
\begin{equation}\label{eq:3}
H(c) \le c(4\log c + O(\log\log c)) 
\end{equation}
for large $c$; he showed also that $H(c) \le 5$ when $1 < c < 3/2$. These results can be considered analogues of results of L.-K. Hua from the 1930's and the 1940's that dealt with the classical Waring--Goldbach problem. In particular, Hua proved the appropriate variant of \eqref{eq:3} for integer $c$. The paper \cite{PS52} went unnoticed for almost forty years  until the work of Tolev \cite{To92} that established the bound $H(c) \le 3$ for $1 < c < 15/14$. The latter result has motivated a series of improvements \cite{BW14,CZ02,Ku09,KL02,KN98} culminating in the recent result of Baker and Weingartner \cite{BW14} that $H(c) \le 3$ for $1 < c < 10/9$. There has also been further work \cite{BW13, CZ03, CZ07, Ga03, LZ16, LS13} on extending the range of $c$ in Piatetski-Shapiro's result on sums of five powers of primes: the best result in that direction,  also due to Baker and Weingartner \cite{BW13}, states that $H(c) \le 5$ for $1 < c \le 2.041$, $c \ne 2$.

Note that the sequence of Piatetski-Shapiro primes of index $\gamma$ is a ``thin'' sequence of primes (and gets thinner as $\gamma$ decreases). As researchers in additive prime number theory have asked whether different additive questions about  the primes can be resolved in prime numbers from thin sets, Piatetski-Shapiro primes have become a favorite ``test case'': see \cite{AG17,AG18, BF92,Ji95,Ku97,Li03,Mi16,Pe03,ZZ05} for some results on solubility of classical additive problems in Piatetski-Shapiro primes. In the present note, motivated by recent work on solubility of Diophantine inequalities in primes from special sequences (for example, \cite{Di17b,Di17a,To17}), we study the solubility of the Diophantine inequality \eqref{eq:i.2} in Piatetski-Shapiro primes. Our main result is the following.

\begin{theorem}\label{thm1}
Let  $c > 5$, $c \notin \mathbb N$,  and  $1-\rho < \gamma < 1$, where $\rho = (8c^2 + 12c + 12)^{-1}$. Then for
$s \ge 4c \log c + \frac 43c + 10$ and sufficiently large $N$, the inequality
\begin{equation}\label{eq:thm1}
\big| p_1^c + p_2^c + \dots + p_s^c - N \big| < (\log N)^{-1}     
\end{equation} 
has solutions in prime numbers $p_1, p_2, \dots, p_s \in \mathcal N_\gamma$. 
\end{theorem}

We remark that this theorem represents also a slight improvement on Piatetski-Shapiro's original bound \eqref{eq:3} for $H(c)$. In that regard, the bound on $s$ in Theorem \ref{thm1} can be compared with recent results by Wooley and the first author \cite{KW16, KW17}, who obtained similar improvements on the aforementioned result of Hua on the classical Waring--Goldbach problem. Such a comparison suggests that one may be able to further reduce the upper bound on $H(c)$ by establishing analogues for Diophantine inequalities of some technical lemmas from \cite{KW16, KW17} that count solutions of Diophantine equations with variables in diminishing ranges. We do not pursue such improvements here, since our main focus is on the hybrid nature of our results, but we intend to return to this aspect of the problem in the future.

We study also the solubility of the ternary inequality \eqref{eq:thm1} in Piatetski-Shapiro primes and establish the following variant of Tolev's result in \cite{To92}.

\begin{theorem}\label{thm2}
Let $\gamma < 1 < c$ and $15(c-1) + 28(1-\gamma) < 1$. Then for sufficiently large $N$, the inequality
\[ \big| p_1^c + p_2^c + p_3^c - N \big| < (\log N)^{-1} \]
has solutions in prime numbers $p_1, p_2, p_3 \in \mathcal N_\gamma$. 
\end{theorem}

We remark that the ranges of $\gamma$ and $c$ in this result can possibly be extended by an appeal to more sophisticated exponential sum estimates. However, since the resulting improvement is not likely to be great, we have chosen not to pursue such matters. On the other hand, the proof of Theorem \ref{thm2} can be easily adapted to establish the following companion results on the binary and quaternary inequalities. 

\begin{theorem}\label{thm4}
Let $\gamma < 1 < c$ and $8(c-1) + 21(1-\gamma) < 1$. For a large $Z$, let $\mathcal E(Z)$ denote the set of $N \in (Z/2, Z]$ for which the inequality
\[ \big| p_1^c + p_2^c - N \big| < (\log N)^{-1} \]
has no solutions in prime numbers $p_1, p_2 \in \mathcal N_\gamma$. Then the Lebesgue measure of $\mathcal E(Z)$ is $O(Z\exp(-(\log Z)^{1/4}))$.  
\end{theorem}

\begin{theorem}\label{thm3}
Let $\gamma < 1 < c$ and $8(c-1) + 21(1-\gamma) < 1$. Then for sufficiently large $N$, the inequality
\[ \big| p_1^c + p_2^c + p_3^c + p_4^c - N \big| < (\log N)^{-1} \]
has solutions in prime numbers $p_1, p_2, p_3, p_4 \in \mathcal N_\gamma$. 
\end{theorem}

\paragraph*{\em Notation.}
In this paper, $p, p_{1}, \dots $ will always denote primes. We also reserve $\varepsilon$ for a fixed small positive number that can be chosen arbitrarily small; its value need not be the same in all occurrences. As usual in analytic number theory, Vinogradov's notation $A \ll B$ means that $A = O(B)$, and we write $A \asymp B$ if $A \ll B \ll A$. Sometimes we use $x \sim X$ as an abbreviation for  $x \in (X/2, X]$.

We write $e(x) = e^{2\pi i x}$ and $\Psi_\gamma(n) = \psi(-(n+1)^\gamma ) - \psi( -n^\gamma )$, with $\psi(x) = x - \lfloor x \rfloor - 1/2$, and we define $(\alpha)_s$ recursively by $(\alpha)_0 = 1$ and $(\alpha)_s = (\alpha)_{s-1}(\alpha-s+1)$ for $s \ge 1$. We also write $\mathcal N_{\gamma}(X) = \mathcal N_{\gamma} \cap (X/2,X]$ and define several generating functions:
\begin{align*}
S(\theta; X)   & = \sum_{p \in \mathcal N_\gamma(X)} (\log p)e(\theta p^c), & T(\theta; X)   &= \sum_{n \in \mathcal N_\gamma(X)} e(\theta n^c), \\
S_{0}(\theta; X) & = \sum_{p \sim X} \gamma p^{\gamma - 1}(\log p) e(\theta p^c), & T_{0}(\theta; X) &= \sum_{n \sim X} \gamma n^{\gamma - 1} e(\theta n^c), \\
S_{1}(\theta; X) & = \sum_{p \sim X} \Psi_\gamma(p)(\log p)e(\theta p^c), & T_{1}(\theta; X) & = \sum_{n \sim X} \Psi_\gamma(n)e(\theta n^c), 
\end{align*}
\[ V(\theta; X)    = \gamma \int_{X/2}^X u^{\gamma-1} e( \theta u^c ) \, du. \] 

\section{Lemmas}

\begin{lemma}\label{l1}
Let $(a_n)$ be a sequence of complex numbers with $|a_n| \le A$. Then
\begin{equation*}
\sum_{n \in \mathcal N_{\gamma}(X)} a_n = \gamma \sum_{n \sim X} a_nn^{\gamma-1} + \sum_{n \sim X} a_n \Psi_\gamma(n) + O\big( A X^{\gamma-1} \big).
\end{equation*}
\end{lemma}

\begin{proof}
This is immediate on noting that the indicator function of the set $\mathcal N_\gamma$ can be expressed as
\[ \big\lfloor -n^\gamma \big\rfloor - \big\lfloor -(n+1)^\gamma \big\rfloor = (n+1)^\gamma - n^\gamma + \Psi_\gamma(n). \qedhere \]
\end{proof}

In particular, Lemma \ref{l1} yields
\begin{equation}\label{2.1}
S(\theta; X) = S_0(\theta; X) + S_1(\theta; X) + O(1),
\end{equation}
and
\begin{equation}\label{2.1b}
T(\theta; X) = T_0(\theta; X) + T_1(\theta; X) + O(1).
\end{equation}

\begin{lemma}\label{l3}
Let $(a_n)$ be a sequence of complex numbers with $|a_n| \le A$. When $0 < \sigma < (2\gamma-1)/3$ and $X^{1 - \gamma + \sigma} \le H \le X^{\gamma - 2\sigma}$, one has
\[ \sum_{n \sim X} a_n\Psi_\gamma(n) \ll \sup_{\substack{Y \sim X\\ u \in \{0,1\}}} \sum_{1 \le |h| \le H} \Phi(h) \bigg| \sum_{Y < n \le X} a_ne(h(n+u)^\gamma) \bigg|+ AX^{\gamma - \sigma}, \]
where $\Phi(h) = \min\big( X^{\gamma-1}, |h|^{-1} \big)$.
\end{lemma}

\begin{proof}
We follow closely the proof of Theorem 4.11 in \cite{GK91}.
Starting with Vaaler's approximation to $\psi$ (Theorem A.6 in \cite{GK91}), the argument on pp. 47--48 in \cite{GK91} yields
\[ \sum_{n \sim X} a_n\Psi_\gamma(n) \ll \sum_{1 \le |h| \le H} \frac 1{|h|} \bigg| \sum_{n \sim X} a_n\big( e(hn^\gamma) - e(h(n+1)^\gamma) \big) \bigg|+ AX^{\gamma - \sigma}. \]
A partial summation argument similar to that on p. 49 in \cite{GK91} then shows that the last sum is bounded by
\[ X^{\gamma-1} \sup_{Y \sim X} \sum_{1 \le |h| \le H} \bigg| \sum_{Y < n \le X} a_n e(hn^\gamma) \bigg|. \]
On the other hand, by the triangle inequality, the same sum is bounded by
\[ \max_{u \in \{0,1\}} \sum_{1 \le |h| \le H} \frac 1{|h|} \bigg| \sum_{n \sim X} a_n e(h(n+u)^\gamma) \bigg|. \]
\end{proof}

\begin{lemma}\label{lem3.r}
Let $F, N$ be large parameters, $N \le N_1 \le 2N$, and let $r \ge 2$ be an integer. Suppose that $f : [N,N_1] \to \mathbb R$ has $r$ continuous derivatives and satisfies
\[ FN^{-r} \ll \left| f^{(r)}(x) \right| \ll FN^{-r} \qquad (N \le x \le N_1). \]
Then one has 
\begin{equation}\label{eq:lem3.m} 
\sum_{N < n \le N_1} e(f(n)) \ll N^{1+\eps} \left( (FN^{-r})^{\nu} + N^{-\nu} + F^{-2\nu/r} \right), 
\end{equation}
where $\nu = \nu_r = (r^2-r)^{-1}$. 
\end{lemma}

\begin{proof}
The case $r=2$ of \eqref{eq:lem3.m} is classical: see Theorem 2.2 in \cite{GK91}, for example. When $r \ge 3$, the bound is a version of Theorem 1 in Heath-Brown \cite{HB17}. 
\end{proof}

\begin{lemma}\label{lem3.3}
Suppose that the hypotheses of Lemma \ref{lem3.r} hold for $r=3$. If $F \ge N$, one has 
\begin{equation}\label{eq:lem3.i} 
\sum_{N < n \le N_1} e(f(n)) \ll F^{1/6}N^{1/2} + N^{3/4}. 
\end{equation}
Moreover, if $|f''(x)| \ll FN^{-2}$ for all $x \in [N,N_1]$, one has
\begin{equation}\label{eq:lem3.ii} 
\sum_{N < n \le N_1} e(f(n)) \ll F^{1/6}N^{1/2} + NF^{-1/3}. 
\end{equation}
\end{lemma}

\begin{proof}
Bounds like \eqref{eq:lem3.i} are well-known: the above version follows from Theorem 2.6 in \cite{GK91}. When $F \le N^{3/2}$, inequality~\eqref{eq:lem3.ii} follows from Lemma 1 in Kumchev \cite{Ku09}; otherwise, it follows from \eqref{eq:lem3.i}. 
\end{proof}

In the remainder of this section, we apply the above general bounds to exponential sums with phase functions derived from $f(x) = \theta x^c + h(x + u)^\gamma$, where $u \in \{0,1\}$ and $\theta, h$ are real parameters. 

\begin{lemma}\label{lem5}
Let $1/2 < \gamma < 1 < c$, $|h| \le X^{4/3-\gamma}$, and $X^{\gamma-c} \le |\theta| \le X^{\delta}$ for a sufficiently small $\delta > 0$. Then, for $u \in \{0,1\}$ and $Y \sim X$, one has
\[ \sum_{Y < n \le X} e(\theta n^c + h(n+u)^{\gamma}) \ll X^{1-\nu}, \]
where $\nu = (c^2+3c+2)^{-1}$.
\end{lemma}

\begin{proof}
Let $X^{\alpha} = |\theta|X^c$ and $F = X^{\alpha} + |h|X^{\gamma}$, and write $f(x) = \theta x^c + h(x + u)^{\gamma}$. We consider two cases depending on the size of $\alpha$. 

\smallskip

{\sl Case 1:} $\alpha \ge 3/2$. Then we have 
\[ \left| f^{(r)}(x) \right| \asymp X^{\alpha-r}, \]
and hence, Lemma \ref{lem3.r} with $r = \lceil \alpha \rceil + 1$ yields
\[ \sum_{Y < n \le X} e(f(n)) \ll X^{1 - \nu_r + \eps} \ll X^{1 - \nu}, \]
where $\nu_r = (r^2 - r)^{-1}$.

{\sl Case 2:} $\gamma \le \alpha \le 3/2$. Note that in this case we have $X^{1/2} \le F \le X^{3/2}$. We can split the interval $(Y, X]$ into at most three subintervals such that on each of them 
\[ \left| f^{(r)}(x) \right| \asymp FX^{-r}, \]
holds for $r=2$ or $3$. Moreover, we always have $|f''(x)| \ll FX^{-2}$. Thus, combining \eqref{eq:lem3.ii} and the case $r=2$ of \eqref{eq:lem3.m}, we get 
\[ \sum_{Y < n \le X} e(f(n)) \ll F^{1/2}  + F^{1/6}X^{1/2} + XF^{-1/3} \ll X^{5/6}. \]
\end{proof}

\begin{corollary}\label{lem:T}
Let $1 - \nu < \gamma < 1 < c$, with $\nu = (c^2+3c+2)^{-1}$, and $X^{\gamma-c} \le |\theta| \le X^{\delta}$ for a sufficiently small $\delta > 0$. Then one has
\[ T(\theta; X) \ll X^{1-\nu+\eps}. \]
\end{corollary}

\begin{proof}
By Lemma \ref{lem5} with $h = 0$ and partial summation, we have
\begin{equation}\label{eq:9} 
T_0(\theta; X) \ll X^{\gamma - \nu}. 
\end{equation}
Using \eqref{2.1b}, \eqref{eq:9} and Lemma \ref{l3} with $\sigma = \nu + \gamma - 1$ and $H = X^{\nu}$, we reduce the corollary to the bound
\[ 
\sup_{\substack{Y \sim X\\ u \in \{0,1\}}} \sum_{1 \le |h| \le H} \Phi(h) \left| \sum_{Y < n \le X} e(\theta n^c + h(n+u)^\gamma ) \right| \ll X^{1 - \nu + \eps}, \]
which is an immediate consequence of Lemma \ref{lem5}. 
\end{proof}

Next, we establish similar estimates for $S_1(\theta; X)$ and $S_0(\theta; X)$. As usual, we derive our estimates from bounds on  double sums of the form
\[ \sum_{m \sim M} \sum_{Y < mk \le X} a_m b_k e(\theta (mk)^c + h(mk + u)^\gamma), \]
where $u \in \{0,1\}$ and $Y \sim X$. 

\begin{lemma}\label{lem:T1a}
Let $c > 5$, $3/4 < \gamma < 1$, $|h| \le X^{4/3-\gamma}$, and $X^{\gamma-c-\delta} \le |\theta| \le X^{\delta}$ for a sufficiently small $\delta > 0$. Also, let $(a_m)$ be a sequence of complex numbers such that $|a_m| \ll 1$, and suppose that 
\begin{equation}\label{eq:M1}
M \ll X^{1/2 + \rho},    
\end{equation}
where $\rho = (8c^2 + 12c + 12)^{-1}$. Then, for $u \in \{0,1\}$ and $Y \sim X$, one has
\begin{align*}
\sum_{m \sim M} \sum_{Y < mk \le X} a_m e(\theta (mk)^c + h(mk + u)^\gamma) \ll X^{1 - \rho + \eps}. 
\end{align*}
\end{lemma}

\begin{proof}
Let $X^{\alpha} = |\theta|X^c$, $F = X^{\alpha} + |h|X^{\gamma}$, $K = X/M$, and write $f_m(x) = \theta (mx)^c + h(mx + u)^{\gamma}$. We consider two cases depending on the size of $\alpha$. 

\smallskip

{\sl Case 1:} $\alpha \ge 4/3 + \delta$. Then we have
\[ \left| f_m^{(r)}(x) \right| \asymp X^{\alpha}K^{-r} \]
for all $r \in \mathbb N$ and all $x$ with $Y < mx \le X$. We choose $r \in \mathbb N$ so that $K^{r-2} \le X^{\alpha} < K^{r-1}$. Since $r \ge 3$, we can apply Lemmas \ref{lem3.r} or \ref{lem3.3} to get 
\begin{equation}\label{eq:lem4.1}
\sum_{Y < mk \le X} e( f_m(k) ) \ll K^{1-\nu+\eps},
\end{equation}
where $\nu = (r^2-r)^{-1}$. From \eqref{eq:M1} and the definitions of $\alpha, \rho$ and $r$, we deduce that
\begin{align*} r(r-1) &\le \left( \frac {2\alpha}{1-2\rho} + 2 \right)\left( \frac {2\alpha}{1-2\rho} + 1 \right) \\
&< \left( \frac {8c^3}{4c^2-1} + 2 \right)\left( \frac {8c^3}{4c^2-1} + 1 \right) < 4c^2 + 6c + 4.5, 
\end{align*}
provided that $c > 5$ and $\delta$ is sufficiently small. Hence, $K^{-\nu} \ll X^{-\rho}$ and the desired bound follows from \eqref{eq:lem4.1}.

{\sl Case 2:} $\gamma - \delta\le \alpha \le 4/3 + \delta$. In this case, we have $F < K^3$, and so we can choose $r \in \{2,3,4\}$ with $K^{r-2} \le F < K^{r-1}$. We remark that if the inequality
\begin{equation}\label{eq:sth_deriv} 
\left| f_m^{(s)}(x) \right| \asymp FK^{-s}
\end{equation}
with $s=r$ fails for any $x$ with $Y < mx \le X$, then those exceptional $x$ belong to a subinterval of $(Ym^{-1}, Xm^{-1}]$ where \eqref{eq:sth_deriv} holds with $s=r+1$. Therefore, when $K^{r-2} \le F \le K^{r-1}$, we can combine the cases $r$ and $r+1$ of \eqref{eq:lem3.m} to show that
\[ \sum_{Y < mk \le X} e( f_m(k) ) \ll K^{1+\eps}\big( K^{-\nu} + F^{-2\nu/(r+1)} \big), \]
where $\nu = \nu_{r+1} = (r^2 + r)^{-1}$.
When $r = 3$ or $4$, this leads to the bound \eqref{eq:lem4.1} with $\nu = 1/24$ and $\nu = 1/25$, respectively. When $r=2$, we recall that $F \ge X^{\gamma-\delta}$, and hence,
\[ \sum_{Y < mk \le X} e( f_m(k) ) \ll K^{1+\eps}F^{-1/9} \ll K^{1-\gamma/9+\delta} \ll K^{11/12}. \]
We conclude that in the present case, \eqref{eq:lem4.1} holds with $\nu = 1/25$, which more than suffices to deduce the claim of the lemma.
\end{proof}

\begin{lemma}\label{lem:T2a}
Let $c > 5$, $1 - 2\rho < \gamma < 1$, with $\rho = (8c^2 + 12c + 12)^{-1}$, $|h| \le X^{4/3-\gamma}$, and $X^{\gamma-c-\delta} \le |\theta| \le X^{\delta}$ for a sufficiently small $\delta > 0$. Also, let $(a_m)$ and $(b_k)$ be sequences of complex numbers such that $|a_m| \ll 1$ and $|b_k| \ll 1$, and suppose that 
\begin{equation}\label{eq:M2}
X^{2\rho} \le M \ll X^{1/3}.    
\end{equation}
Then, for $u \in \{0,1\}$ and $Y \sim X$, one has
\begin{align*}
\sum_{m \sim M} \sum_{Y < mk \le X} a_mb_k e(\theta (mk)^c + h(mk + u)^\gamma) \ll X^{1 - \rho + \eps}. 
\end{align*}
\end{lemma}

\begin{proof}
Let $W$ denote the double sum in question, and write $f_m(x) = \theta (mx)^c + h(mx + u)^\gamma$, $X^{\alpha} = |\theta|X^{c}$, $F = X^{\alpha} + |h|X^{\gamma}$, and $K = X/M$. By Cauchy's inequality, 
\[ |W|^{2} \ll K \sum_{K/2 < k \le 2K} \bigg| \sum_{\substack{m \sim M\\ Y < mk \le X}} {a_m e( f_m(k) )} \bigg|^{2}. \]
Let $Q = X^{2\rho-\eps}$. The Weyl-van der Corput lemma (Lemma~2.5 in \cite{GK91}) gives
\begin{equation} \label{eq:lemT2a.1} 
|W|^{2} \ll \frac{X}{Q} \sum_{K/2 < k \le 2K}\sum_{|q|\leq Q}{\left ( 1 - \frac{|q|}{Q}\right )} \sum_{m \in \mathcal I} a_{m + q}\overline{a_m} \, e(g_{q,m}(k)), 
\end{equation}
where $g_{q,m}(x) = f_{m+q}(x) - f_m(x)$ and $\mathcal I$ is the subinterval of $(M/2,M]$ defined by the inequalities
\begin{equation} \label{eq:lemT2a.2}
Y < km \le X, \quad Y < k(m+q) \le X. 
\end{equation}
We estimate the contribution from terms with $q=0$ to the sum in \eqref{eq:lemT2a.1} trivially. We change the order of summation in the remainder of that sum to obtain
\begin{equation}\label{eq:16}
|W|^{2} \ll \frac{X^{2}}{Q} + \frac{X}{Q}\sum_{1 \le |q| \leq Q}
\sum_{m \sim M} \bigg| \sum_{k \in \mathcal I'} e(g_{q,m}(k)) \bigg|,
\end{equation}
where $\mathcal I'$ is the subinterval of $(K/2,2K]$ subject to conditions \eqref{eq:lemT2a.2}. Similarly to the proof of Lemma \ref{lem:T1a}, the estimation of the last sum breaks into two cases.

\smallskip

{\sl Case 1:} $\alpha \ge 4/3 + \delta$. In this case, we have
\[ \left| g_{q,m}^{(r)}(x) \right| \asymp |q|X^{\alpha-1}K^{1-r} \]
for all $r \in \mathbb N$ and all $x \in \mathcal I'$. We choose $r$ so that $K^{r-3} \le X^{\alpha-1+2\rho} < K^{r-2}$ and apply Lemmas \ref{lem3.r} and \ref{lem3.3} to obtain
\begin{equation}\label{eq:lemT2a.4}
\sum_{k \in \mathcal I'} e( g_{q,m}(k) ) \ll K^{1-\nu+\eps} \big( 1 + (|q|X^{-2\rho})^{-\beta} \big),
\end{equation}
where $\nu = (r^2-r)^{-1}$ and $\beta = \beta_r \in (0, 1/3)$. We insert this bound into the right side of \eqref{eq:16} and sum the result over $m$ and $q$ to get
\begin{equation}\label{eq:lemT2a.5}
|W|^{2} \ll X^{2-2\rho+\eps} + X^{2+\eps}K^{-\nu}.
\end{equation}
From \eqref{eq:M2} and the definitions of $\alpha, \rho$ and $r$, we find that
\[ r(r-1) \le \left( \frac {3\alpha}2 + \frac 32 + 3\rho \right)\left( \frac {3\alpha}2 + \frac 12 + 3\rho \right) < \frac {9c^2}4 + 3c + 2, \]
provided that $c > 5$ and $\delta$ is sufficiently small. Hence, $K^{-\nu} \ll X^{-2\rho}$ and the claim of the lemma follows from \eqref{eq:lemT2a.5}.

{\sl Case 2:} $\gamma - \delta \le \alpha \le 4/3+\delta$. In this case, we have $1 \le FX^{-1+2\rho} \le K$. Similarly to Case~2 of the proof of Lemma \ref{lem:T1a}, we can split $\mathcal I'$ into at most three subintervals so that on each of them we have
\[ \left| g_{q,m}^{(r)}(x) \right| \asymp |q|FX^{-1}K^{1-r} \]
with $r=3$ or $r=4$. Thus, the cases $r = 3$ and $r = 4$ of \eqref{eq:lem3.m} give
\[ \sum_{k \in \mathcal I'} e( g_{q,m}(k) ) \ll K^{23/24+\eps}\big( 1 + (|q|X^{-2\rho})^{-1/9} \big). \]
Combining this bound, \eqref{eq:M2} and \eqref{eq:16}, we deduce that 
\[ |W|^{2} \ll X^{2-2\rho+\eps} + X^{2+\eps}K^{-1/24} \ll X^{2-2\rho+\eps}. \qedhere \]
\end{proof}

\begin{lemma}\label{lem:S1a}
Let $c > 5$, $1 - \rho < \gamma < 1$, with $\rho = (8c^2 + 12c + 12)^{-1}$, and $X^{\gamma-c-\delta} \le |\theta| \le X^{\delta}$ for a sufficiently small $\delta > 0$. Then, one has
\[ S(\theta; X) \ll X^{1 - \rho + \eps}. \]
\end{lemma}

\begin{proof}
By \eqref{2.1}, it suffices to show that
\[ S_j(\theta; X) \ll X^{1-\rho+\eps} \qquad (j=0,1). \]
Lemma \ref{l3} with $\sigma = \rho + \gamma - 1$ and $H = X^{\rho}$ (and an obvious choice of the coefficients $(a_n)$) reduces the estimate for $S_1(\theta;X)$ to the bound
\[ \sum_{1 \le |h| \le H} \Phi(h) \left| \sum_{Y < p \le X} (\log p) e(\theta p^c + h(p + u)^\gamma) \right| \ll X^{1 - \rho + \eps}, \]
where $u \in \{0,1\}$ and $Y \sim X$. Thus, it suffices to show that
\begin{equation}\label{eq:19}
\sum_{Y < n \le X} \Lambda(n) e(\theta n^c + h(n + u)^\gamma) \ll X^{1 - \rho + \eps}
\end{equation}
for all $h$ with $1 \le |h| \le H$ (here, $\Lambda(n)$ is the von Mangoldt function). Since the desired estimate for $S_0(\theta; X)$ follows from \eqref{eq:19} with $h=0$, it remains to establish \eqref{eq:19} for all $h$ with $|h| \le H$.

Let $u, v, z$, with $z - \frac 12 \in \mathbb N$, be parameters to be chosen momentarily subject to the constraints
\begin{equation}\label{eq:S1a.2} 
3 < u < v < z, \quad z \ge 4u^2, \quad 64uz^2 \le x < 32x \le v^3. 
\end{equation}
A combinatorial lemma due to Heath-Brown (see Lemma 3 in \cite{HB83}) decomposes the sum in \eqref{eq:19} into a linear combination of $O(\log^8 X)$ double sums 
\[ \sum_{m \sim M} \sum_{Y < mk \le X} a_m b_k e(\theta (mk)^c + h(mk + u)^\gamma) \]
of two types:
\begin{itemize}
\item Type I: where $a_m \ll m^{\eps}$, $b_k = 1$ or $\log k$, and $M \le Xz^{-1}$;
\item Type II: where $a_m \ll m^{\eps}$, $b_k \ll k^{\eps}$, and $u \le M \le v$.
\end{itemize}
If we choose $u = X^{2\rho}$, $v = 4X^{1/3}$ and $z = \lfloor \frac 1{10} X^{1/2-\rho} \rfloor - \frac 12$, conditions \eqref{eq:S1a.2} are satisfied and we can appeal to Lemmas \ref{lem:T1a} and \ref{lem:T2a} to estimate all Type I and Type II sums and to complete the proof of \eqref{eq:19}.
\end{proof}

\begin{lemma}\label{lem:S1b}
Let $11/12 < \gamma < 1 < c$ and $|\theta| \le X^{\gamma-c-\delta}$ for a sufficiently small $\delta > 0$. Then one has
\[ S_1(\theta; X) \ll X^{11/12+\delta}. \]
\end{lemma}

\begin{proof}
Similarly to the proof of Lemma \ref{lem:S1a}, it suffices to show that
\begin{equation}\label{eq:S1b} 
\sum_{1 \le |h| \le H} \Phi(h) \left| \sum_{Y < p \le X} (\log p) e(\theta p^c + h(p + u)^\gamma) \right| \ll X^{11/12+\delta/2},  
\end{equation}
where $u \in \{0,1\}$, $Y \sim X$ and $H = X^{1-\gamma+\delta}$. The case $\theta = u = 0$ of \eqref{eq:S1b} is essentially a special case of the main part of the proof of Theorem~4.14 in \cite{GK91}: see pp. 50--53 in \cite{GK91} in the case of the exponent pair $(\frac 12, \frac 12)$. The more general bound required here can be established using an identical argument, since under the hypotheses on $u$ and $\theta$, we have
\[ \left| \frac {d^j}{dx^j}\big( \theta x^c + h(x + u)^\gamma \big) \right| \asymp \left| \frac {d^j}{dx^j}\big(  hx^\gamma  \big) \right| \qquad (j=1,2) \]
whenever $1 \le |h| \le H$ and $x \sim X$.
\end{proof}

\begin{lemma}\label{lem:S0}
Let $6\rho < \gamma < 1 < c < 3/2 - 6\rho$ and $X^{\gamma-c-\delta} \le |\theta| \le X^{\delta}$ for a sufficiently small $\delta > 0$ and a fixed $\rho \in (0, 1/12)$. Then one has
\[ S_0(\theta; X) \ll X^{\gamma - \rho + \delta}. \]
\end{lemma}

\begin{proof}
Suppose that $Y \sim X$. The calculations in Lemma 10 of \cite{To92} establish the inequality
\[ \sum_{Y < p \le X} (\log p) e(\theta p^c) \ll |\theta|^{1/2}X^{(2c + 1)/4+\eps} + |\theta|^{1/6}X^{(2c+9)/12 + \eps} + X^{1-\gamma/6+\delta}. \]
Under the hypotheses of the lemma, the stated bound for $S_0(\theta; X)$ follows by partial summation.
\end{proof}

\begin{lemma}\label{lem:T1b}
Let $3\rho < \gamma < 1 < c$, $1 \le |h| \le X^\rho$, and $X^{\gamma-c-\delta} \le |\theta| \le X^{\delta}$, with $0 < \rho < c/6$ and a sufficiently small $\delta > 0$. Also, let $(a_m)$ be a sequence of complex numbers such that $|a_m| \ll 1$, and suppose that 
\begin{equation}\label{eq:M1a}
M \ll \min \big( X^{1-(c+\delta)/2 - \rho}, X^{1-(\gamma+3\rho)/2} \big).   
\end{equation}
Then, for $u \in \{0,1\}$ and $Y \sim X$, one has
\begin{align*}
\sum_{m \sim M} \sum_{Y < mk \le X} a_m e(\theta (mk)^c + h(mk + u)^\gamma) \ll X^{1 - \rho}.
\end{align*}
\end{lemma}

\begin{proof}
Let $K = X/M$ and $F = |\theta|X^c + |h|X^{\gamma}$ and write $f_m(x) = \theta(mx)^c + h(mx + u)^\gamma$. As in the second cases of the proofs of Lemmas \ref{lem5}, \ref{lem:T1a} and \ref{lem:T2a}, we note that when $Y < mx \le X$, we have $|f_m''(x)| \ll FK^{-2}$ and at least one of the bounds
\[ \left| f_m^{(r)}(x) \right| \asymp FK^{-r} \qquad (r = 2,3). \]
Hence, we can combine \eqref{eq:lem3.ii} and the case $r=2$ of \eqref{eq:lem3.m} to obtain
\begin{equation}\label{eq:23}
\sum_{Y < mk \le X} e( f_m(k) ) \ll F^{1/2} + F^{1/6}K^{1/2} + F^{-1/3}K.
\end{equation}
Moreover, since we only need to refer to \eqref{eq:lem3.ii} when $|\theta|X^c \asymp |h|X^{\gamma}$, the middle term in \eqref{eq:23} can be replaced by $|\theta|^{1/6}X^{(c+3)/6}$. Thus,
\begin{equation}\label{eq:24} 
\sum_{m \sim M} \sum_{Y < mk \le X} e( f_m(k) ) \ll F^{1/2}M + (|\theta|X^c)^{1/6}(XM)^{1/2} + X^{1-\gamma/3}. 
\end{equation}
The lemma follows from this bound and the hypotheses on $\rho, \theta, h, M$.
\end{proof}

\begin{lemma}\label{lem:T2b}
Let $\frac 12 + 4\rho < \gamma < 1 < c$, $1 \le |h| \le X^\rho$, and $X^{\gamma-c-\delta} \le |\theta| \le X^{\delta}$, with a sufficiently small $\delta > 0$ and $\rho > 0$ that satisfies the conditions
\begin{equation}\label{eq:25}  
c + 14\rho < 2, \quad 2\gamma + 14\rho < 3, \quad 2c + 12\rho < 3. 
\end{equation}
Also, let $(a_m)$ and $(b_k)$ be sequences of complex numbers such that $|a_{m}| \ll 1$ and $|b_{k}| \ll 1$, and suppose that 
\begin{equation}\label{eq:M.2} 
X^{2\rho} \le M \le X^{1-2\rho}. 
\end{equation}
Then, for $u \in \{0,1\}$ and $Y \sim X$, one has
\begin{equation}
\mathop{\sum_{m \sim M} \sum_{k \sim K}}_{Y < mk \le X} a_mb_k e(\theta (mk)^c + h(mk + u)^\gamma) \ll X^{1 - \rho+\delta}.
\end{equation}
\end{lemma}

\begin{proof}
Let $W$ denote the double sum in question and $F = |\theta|X^c + |h|X^{\gamma}$. By symmetry, we may assume that $M \le K$; hence, $M \ll X^{1/2}$. Similarly to \eqref{eq:16}, we have 
\begin{equation}\label{eq:27}
|W|^{2} \ll \frac{X^{2}}{Q} + \frac{X}{Q}\sum_{1 \le |q| \leq Q}
\sum_{m \sim M} \bigg| \sum_{mk \in \mathcal I} e(g_{q,m}(k)) \bigg|,
\end{equation}
where $Q = X^{2\rho-\eps}$, $\mathcal I$ is a subinterval of $(Y,X]$, and 
\[ g_{q,m}(x) = \theta((m+q)^c - m^c)x^c + h\big( ((m+q)x+u)^\gamma - (mx + u)^\gamma \big). \]
Similarly to \eqref{eq:23} and \eqref{eq:24}, we have
\[ \sum_{Y < mk \le X} e( g_{q,m}(k) ) \ll G_q^{1/2} + (\Delta_q|\theta|X^c)^{1/6}K^{1/2} + G_q^{-1/3}K, \]
where $\Delta_q = |q|M^{-1}$ and $G_q = \Delta_qF$. We insert this bound into the right side of \eqref{eq:27} to deduce that
\begin{align*}
W^2 \ll X^2Q^{-1} + X^{5/4}(FQ)^{1/2} + X^{5/3}(|\theta|X^cQ)^{1/6} + X^{(13-2\gamma)/6}Q^{-1/3}. 
\end{align*}
The lemma follows from the last inequality and the hypotheses on $\rho, \theta$ and $h$. 
\end{proof}

\begin{lemma}\label{lem:S1}
Let $1-\rho < \gamma < 1 < c$ and $X^{\gamma-c-\delta} \le |\theta| \le X^{\delta}$, with a sufficiently small $\delta > 0$ and $\rho > 0$ that satisfies conditions \eqref{eq:25}. Then, one has
\[ S_1(\theta; X) \ll X^{1 - \rho+\delta}. \]
\end{lemma}

\begin{proof}
Similarly to the proof of Lemma \ref{lem:S1a}, it suffices to establish 
\eqref{eq:19} for $u \in \{0,1\}$, $Y\sim X$ and all $h$ with $1 \le |h| \le X^{\rho}$. As in that proof, we decompose the sum in \eqref{eq:19} into double sums. 
We use Vaughan's identity (Lemma~4.12 in~\cite{GK91} with $u = v$) to reduce \eqref{eq:19} to the estimation of $O(\log X)$ sums of the forms
\begin{equation}\label{eq:S.3} \sum_{m \sim M} \sum_{Y < mk \le X} a_m b_k e(\theta (mk)^c + h(mk + u)^\gamma)  
\end{equation}
of Types I and II (here, $U \le X^{1/2}$ is a parameter to be chosen shortly):
\begin{itemize}
\item Type I: where $a_m \ll m^{\eps}$, $b_k = 1$ or $\log k$, and $M \le U^2$;
\item Type II: where $a_m \ll m^{\eps}$, $b_k \ll k^{\eps}$, and $U \le M \le XU^{-1}$.
\end{itemize}

For Type II sums, Lemma \ref{lem:T2b} gives the desired bound under hypotheses \eqref{eq:25} and \eqref{eq:M.2}. Since a Type I sum can be viewed as a special case of a Type II sum, we may estimate a Type I sum using either of Lemmas \ref{lem:T1b} or \ref{lem:T2b}. The ranges \eqref{eq:M1a} and \eqref{eq:M.2} overlap when 
\[ c + 6\rho < 2-\delta, \quad \gamma + 7\rho < 2. \]
Since these inequalities follow from \eqref{eq:25}, we can estimate a Type I sum when
\[ M \le X^{1-2\rho}. \]
Therefore, together Lemmas \ref{lem:T1b} and \ref{lem:T2b} allow us to estimate all the double sums arising from the application of Vaughan's identity, provided that we can choose $u$ with
\[ X^{2\rho} \le U \le X^{1/2 - \rho}. \] 
Since conditions \eqref{eq:25} imply that $\rho < 1/14$, we may choose $U = X^{1/3}$, for example. 
\end{proof}

\begin{corollary}\label{lem:S}
Let $\gamma < 1 < c$, with $15(c-1) + 28(1-\gamma) < 1$, and $X^{\gamma-c-\delta} \le |\theta| \le X^{\delta}$ for a sufficiently small $\delta > 0$. Then one has
\[ S(\theta; X) \ll X^{2\gamma - c - 2\delta}. \]
\end{corollary}

\begin{proof}
By \eqref{2.1}, it suffices to show that
\[ S_j(\theta; X) \ll X^{2\gamma - c - 2\delta} \qquad (j=0,1).\]
Lemma \ref{lem:S0} with $\rho = c - \gamma + 3\delta$ yields the bound on $S_0(\theta; X)$, provided that 
\begin{equation}\label{eq:30}
14(c-1) + 12(1-\gamma) < 1. 
\end{equation} 
We estimate $S_1(\theta;X)$ using Lemma \ref{lem:S1} with $\rho = c+1 - 2\gamma + 3\delta$. With this choice, conditions \eqref{eq:25} can be expressed as 
\[ 15(c-1) + 28(1-\gamma) < 1, \quad 14(c-1)+26(1-\gamma) < 1, \quad 14(c-1) + 24(1-\gamma) < 1. \]
Clearly, the first of these inequalities implies the other two as well as \eqref{eq:30}.
\end{proof}

\begin{corollary}\label{lem:Sa}
Let $\gamma < 1 < c$, with $8(c-1) + 21(1-\gamma) < 1$, and $X^{\gamma-c-\delta} \le |\theta| \le X^{\delta}$ for a sufficiently small $\delta > 0$. Then one has
\[ S(\theta;X) \ll X^{(3\gamma - c)/2 - \delta}. \]
\end{corollary}

\begin{proof}
As in the proof of Corollary \ref{lem:S}, we use \eqref{2.1} and Lemmas \ref{lem:S0}  and \ref{lem:S1}, but we alter the choices of $\rho$: we appeal to Lemma \ref{lem:S0} with $\rho = (c - \gamma)/2 + 2\delta$ and to Lemma \ref{lem:S1} with $\rho = (2+c - 3\gamma)/2 + 2\delta$. With these choices, the application of Lemma \ref{lem:S0} requires that
\[ 8(c-1) + 6(1-\gamma) < 1, \]
and that of Lemma \ref{lem:S1} that 
\[ 8(c-1) + 21(1-\gamma) < 1, \quad 7(c-1)+19(1-\gamma) < 1, \quad 8(c-1) + 18(1-\gamma) < 1. \qedhere \]
\end{proof}

\begin{lemma}\label{lem10}
Let $I$ be an interval in $\mathbb R$. Then one has
\begin{align}
\int_I |S(\theta; X)|^2 \, d\theta & \ll |I|X^{\gamma}L^2 + X^{2\gamma - c}L^3, \label{eq:lem10.0}\\
\int_I |T(\theta; X)|^2 \, d\theta & \ll |I|X^{\gamma} + X^{2\gamma - c}L, \label{eq:lem10.0a}\\
\int_I |S_0(\theta; X)|^2 \, d\theta &\ll |I|X^{2\gamma-1}L + X^{2\gamma-c}L^2, \label{eq:lem10.1}\\
\int_I |V(\theta; X)|^2 \, d\theta &\ll X^{2\gamma-c}L, \label{eq:lem10.2}
\end{align}
where $L = \log X$.
\end{lemma}

\begin{proof}
Consider \eqref{eq:lem10.0}. We have
\begin{align*}
\int_I |S(\theta; X)|^2 \, d\theta &= \sum_{p_1,p_2 \in \mathcal N_\gamma(X)}(\log p_1)(\log p_2) \int_I e(\theta(p_1^c - p_2^c)) \, d\theta \\
&\ll L^2\sum_{n_1, n_2 \in \mathcal N_\gamma(X)} \min\big( |I|, |n_1^c - n_2^c|^{-1} \big) \\
&\ll |I|X^{\gamma}L^2 + X^{1-c}L^2 \sum_{\substack{n_1, n_2 \in \mathcal N_\gamma(X)\\ n_1 < n_2}} (n_2 - n_1)^{-1}, 
\end{align*}
where we have used that $\#\mathcal N_\gamma(X) \ll X^{\gamma}$. We now write $n_i = \big\lfloor m_i^{1/\gamma} \big\rfloor$, with $m_i \sim X^{\gamma}$, and we deduce that 
\begin{align*}
\int_I |S(\theta; X)|^2 \, d\theta &\ll |I|X^{\gamma}L^2 + X^{1-c}L^2 \sum_{\substack{m_1, m_2 \sim X^\gamma\\ m_1 < m_2}} \big( m_2^{1/\gamma} - m_1^{1/\gamma} - 1 \big)^{-1}\\
&\ll |I|X^{\gamma}L^2 + X^{\gamma-c}L^2 \sum_{\substack{m_1, m_2 \sim X^\gamma\\ m_1 < m_2}} (m_2 - m_1)^{-1}\\
&\ll |I|X^{\gamma}L^2 + X^{2\gamma-c}L^3. 
\end{align*}

The proof of \eqref{eq:lem10.0a} is almost identical, and inequalities \eqref{eq:lem10.1} and \eqref{eq:lem10.2} can be proved using similar (and simpler) arguments. The reader can also consult Lemma 7 in \cite{To92} for variants of \eqref{eq:lem10.1} and \eqref{eq:lem10.2}. 
\end{proof}

\section{Proof of Theorem \ref{thm1}}

Let $s = 2t + 2u + 1$, where $t,u$ are integers to be chosen in terms of $c$ in due course. For a large $N$, we set
\begin{equation*}
X = N^{1/c}, \qquad X_0 = (3u)^{-1}X, \qquad X_1 = X, \qquad X_j = {\textstyle \frac 12} X_{j-1}^{1-1/c} \quad (2 \le j \le t).  
\end{equation*}
We use the Davenport--Heilbronn form of the circle method to count the solutions of \eqref{eq:thm1} in primes $p_1, \dots, p_s$ subject to
\begin{equation}\label{eq:1.3}
p_{1}, \dots, p_{2u+1} \in \mathcal N_\gamma(X_0), \qquad p_{2u+2j}, p_{2u+2j+1} \in \mathcal N_\gamma(X_j) \quad (1 \le j \le t).     
\end{equation}
Let us fix a kernel $K \in C^\infty(\mathbb R)$ such that
\begin{equation*}
\widehat K(t) \ge 0, \qquad {\textstyle \frac 14}\mathbf 1_I(4x) \le K(x) \le \mathbf 1_I(x), 
\end{equation*}
where $\mathbf 1_I$ is the indicator function of the interval $I = [-1, 1]$. We can ensure these conditions by choosing $K$ to be a convolution of the form $K = \widetilde{K} \star \widetilde{K}$, where $\widetilde{K} \in C^\infty(\mathbb R)$ is even and satisfies $\mathbf 1_I(4x) \le \widetilde{K}(x) \le \mathbf 1_I(2x)$. We consider the quantity
\begin{equation*}
R(N) = \sum_{p_1, \dots, p_s :  \eqref{eq:1.3}} \bigg\{ \prod_{j=1}^s (\log p_j) \bigg\}  K_\tau(p_1^c + \dots + p_s^c - N),
\end{equation*}
where $K_\tau(x) = K(x/\tau)$, $\tau = (\log N)^{-1}$.  By Fourier inversion,
\begin{equation}\label{eq:2.3}
R(N) = \int_{\mathbb R} \mathcal F_1(\theta)e(-N\theta) \, \dtheta,
\end{equation}
where $\dtheta = \widehat{K_\tau}(\theta) \, d\theta$ and
\[ \mathcal F_j(\theta) = S(\theta; X_0)^{2u+j} S(\theta; X_1)^2 \cdots S(\theta; X_t)^2 \qquad (j=0,1). \]
We analyze the last integral to show that
\begin{equation} \label{eq:46}
R(N) \gg \tau (X_1^2 \cdots X_t^2X^{2u+1})^{\gamma} X^{-c} = \Xi, \qquad \text{say}.
\end{equation}

\subsection{The trivial region}

We first estimate the contribution of large $\theta$ to the integral in \eqref{eq:2.3}. Let $\delta = \delta(c,\gamma) > 0$ be a sufficiently small fixed number. Because of the compact support of the kernel $K$, we have
\[ \widehat{K_\tau}(\theta) = \tau\widehat{K}(\tau\theta) \ll_j \frac {\tau}{(1 + \tau|\theta|)^{j+2}}. \]
Hence, if we fix $j \ge (c+1)\delta^{-1}$, we have
\begin{equation}\label{eq:47}
\int_{|\theta| \ge X^{\delta}} \big| \mathcal F_1(\theta) \big| \, \dtheta \ll \int_{|\theta| \ge X^{\delta}} \frac {\Xi X^c \, d\theta}{(1 + \tau|\theta|)^{j+2}} \ll \Xi X^{-1}.
\end{equation}
 
\subsection{The minor arcs}
The set of ``minor arcs'' is
\[ \mathfrak m = \big\{ \theta : X^{\gamma-c-\delta} \le |\theta| \le X^{\delta} \big\}. \]
From Lemma \ref{lem:S1a}, we have 
\begin{equation}\label{eq:5.1} \sup_{\theta \in \mathfrak m} |S(\theta; X_0)| \ll X^{1-\rho+\eps}, 
\end{equation}
with $\rho = (8c^2 + 12c + 12)^{-1}$. On the other hand,
\[ \int_{\mathfrak m} |\mathcal F_0(\theta)| \, \dtheta \le \int_{\mathbb R} |\mathcal F_0(\theta)| \, \dtheta, \]
and the last integral is bounded by $(\log X)^{s-1}$ times the number of solutions of the Diophantine inequality
\begin{equation}\label{eq:5.4}
\left| x_1^c - x_2^c + x_3^c - x_4^c + \dots + x_{s-2}^c - x_{s-1}^c \right| < \tau
\end{equation}
in integers $x_1, \dots, x_{s-1}$ subject to
\begin{equation}\label{eq:5.5}
x_1, \dots, x_{2u} \in \mathcal N_{\gamma}(X_0), \qquad x_{2u+2j-1}, x_{2u+2j} \in \mathcal N_{\gamma}(X_j) \quad (1 \le j \le t).
\end{equation}
Thus,
\begin{equation}\label{eq:6.1}
\int_{\mathfrak m} |\mathcal F_0(\theta)| \, \dtheta \ll X^{\eps} \int_{\mathbb R} |\mathcal G(\theta)|^2 \, \dtheta[4\tau], 
\end{equation}
where
\[ \mathcal G(\theta) = T(\theta; X_0)^u \, T(\theta; X_1) \cdots T(\theta; X_t). \]

Similarly to \eqref{eq:47}, we have
\begin{equation}\label{eq:6.4}
\int_{|\theta| \ge X^{\delta}} |\mathcal G(\theta)|^2 \, \dtheta[4\tau] \ll \Xi X^{-1-\gamma}. 
\end{equation}
Moreover, applying \eqref{eq:lem10.0a} to $\mathfrak M = (-X^{\gamma-c-\delta}, X^{\gamma - c-\delta})$, we get
\begin{equation}\label{eq:7.1}
\int_{\mathfrak M} |\mathcal G(\theta)|^2 \, \dtheta[4\tau] \ll \Xi X^{c-3\gamma} \int_{\mathfrak M} |T(\theta; X_1)|^2 \, d\theta \ll \Xi X^{-\gamma+\eps}. 
\end{equation}

Finally, we estimate the contribution to the right side of \eqref{eq:6.1} from the minor arcs $\mathfrak m$. By Corollary \ref{lem:T}, we have
\[ \int_{\mathfrak m} |\mathcal G(\theta)|^2 \, \dtheta[4\tau] \ll X^{2u(1-\nu)+\eps} \int_{\mathbb R} \prod_{j=1}^t |T(\theta; X_j)|^2 \, \dtheta[4\tau], \]
where $\nu = (c^2 + 3c + 2)^{-1}$. The latter integral is bounded by the number of solutions of the inequality
\[ \left| x_1^c - x_2^c + x_3^c - x_4^c + \dots + x_{2t-1}^c - x_{2t}^c \right| < 4\tau \]
subject to $x_{2j-1}, x_{2j} \in \mathcal N_{\gamma}(X_j)$, $1 \le j \le t$. By a standard argument (see p. 71 of \cite{Va97}), this inequality has only diagonal solutions (i.e., those with $x_{2j-1} = x_{2j}$ for all $j$). Hence, 
\begin{equation}\label{eq:9.1}
\int_{\mathfrak m} |\mathcal G(\theta)|^2 \, \dtheta[4\tau] \ll X^{2u(1-\nu)+\eps} (X_1 \cdots X_t)^{\gamma} \ll \Xi X^{-\gamma + \Delta + \eps}, 
\end{equation}
where
\begin{align*} 
\Delta &= 2u(1 - \gamma - \nu) + c - \gamma \sum_{j=0}^{t-1}\left( 1 - \frac 1c \right)^j \\
&< (1-\gamma)(2u + c) + \gamma ce^{-t/c} - 2u\nu.
\end{align*}
We choose $t = \lceil 2c\log c \rceil$ and $u = \big\lceil \frac 23c +\frac 12 \big\rceil + 2$, recall that $0 < 1 - \gamma < \rho$, and get
\begin{align*}
 \Delta &< \left( \frac {7c}3 + 7 \right)\rho + \frac 1c - \left(\frac {4c}3+5\right)\nu \\
 &= -\frac {(c-3) (c^3 + 21c^2 + 22c + 24)}{12 c (c+1) (c+2) (2c^2 + 3c + 3)} < 0.   
\end{align*}  

Combining \eqref{eq:6.1}--\eqref{eq:9.1}, we deduce that
\begin{equation*}
\int_{\mathfrak m} |\mathcal F_0(\theta)| \, \dtheta \ll \Xi X^{-\gamma+\eps},
\end{equation*}
and hence, by \eqref{eq:5.1}, 
\begin{equation}\label{eq:55}
\int_{\mathfrak m} \big| \mathcal F_1(\theta) \big| \, \dtheta \ll \Xi X^{-\delta}. 
\end{equation}

\subsection{The major arc}

The major arc of the integral in \eqref{eq:2.3} is the open interval $\mathfrak M$ above. When $\gamma < 1$, $|\theta| < X^{\gamma-c}$ and $Y \sim X$, the argument of Lemma 14 in \cite{To92} yields the approximation
\begin{equation*}
\sum_{Y < p \le X} (\log p)e(\theta p^c) = \int_Y^X e(\theta u^c) \, du + O\big( X^{1-2\eta(X)} \big),
\end{equation*}
where $\eta (X) = (\log X)^{-3/4}$. Using partial summation, we deduce that
\begin{equation}\label{eq:50}
S_0(\theta; X_j) = V(\theta; X_j) + O\big( X_j^{\gamma - 2\eta(X)} \big) 
\end{equation}
for $j=0,1$. When $j \ge 2$, the same approximation follows from the Prime Number Theorem. Combining \eqref{2.1}, \eqref{eq:50} and Lemma \ref{lem:S1b}, we conclude that, for $\theta \in \mathfrak M$, one has 
\begin{equation}\label{eq:51}
S(\theta; X_j) = V(\theta; X_j) + O\big( X_j^{\gamma - 2\eta(X)} \big). 
\end{equation}

Let 
\[ \mathcal F^*(\theta) = V(\theta; X_0)^{2u+1} V(\theta; X_1)^2 \cdots V(\theta; X_t)^2.\]
By \eqref{eq:51}, we have
\[ \mathcal F_1(\theta) - \mathcal F^*(\theta) \ll \tau^{-1} \Xi X^{c-2\gamma - 2\eta(X)} \big( |S(\theta; X_0)|^2 + |V(\theta; X_0)|^2 \big) \] 
for all $\theta \in \mathfrak M$. Recalling \eqref{eq:lem10.1} and \eqref{eq:lem10.2}, we obtain that
\begin{align*}
\int_{\mathfrak M} \big( \mathcal F_1(\theta) - \mathcal F^*(\theta) \big) e(-N\theta) \, \dtheta \ll \Xi X^{c-2\gamma-2\eta(X)} X^{2\gamma - c}(\log X)^3 \ll \Xi X^{-\eta(X)}.
\end{align*}
Since Lemma 3.1 in \cite{GK91} gives
\begin{equation*}
V(\theta; X) \ll X^{\gamma-c}|\theta|^{-1},
\end{equation*}
we get
\begin{equation} \label{eq:58}
\int_{\mathfrak M} \mathcal F_1(\theta) e(-N\theta) \, \dtheta = \int_{\mathbb R} \mathcal F^*(\theta) e(-N\theta) \, \dtheta + O\big( \Xi X^{-\eta(X)} \big).
\end{equation}
Moreover, a standard Fourier integral argument (similar to the proof of Lemma 6 in \cite{To92}, for example) gives 
\begin{equation}\label{eq:59}
\int_{\mathbb R} \mathcal F^*(\theta)e(-N\theta) \, \dtheta \gg \Xi.
\end{equation}
The desired bound \eqref{eq:46} now follows from \eqref{eq:47}, \eqref{eq:55}, \eqref{eq:58} and \eqref{eq:59}.

\section{Proofs of Theorems \ref{thm2}, \ref{thm4} and \ref{thm3}}

\subsection{Theorems \ref{thm2} and \ref{thm3}}

Let $X = (N/2)^{1/c}$, $\tau = \log N$ and consider the quantity 
\[ R(N) = \sum_{p_1, \dots, p_s \in \mathcal N_{\gamma}(X)} \bigg\{ \prod_{j=1}^s (\log p_j) \bigg\}  K_\tau(p_1^c + \dots + p_s^c - N), \]
where $s = 3$ or $4$ and $K_\tau$ is the smooth kernel from the proof of Theorem \ref{thm1}. Let $\delta > 0$ be sufficiently small. Similarly to the proof of Theorem \ref{thm1}, we have
\[ R(N) = \int_{\mathbb R} S(\theta; X)^s e(-N\theta) \, \dtheta \]
and 
\begin{gather} 
\int_{|\theta| \ge X^{\delta}} |S(\theta; X)|^s \, \dtheta \ll \tau X^{s\gamma - c-1}, \notag \\
\int_{\mathfrak M} S(\theta; X)^se(-N\theta) \, \dtheta \gg \tau X^{s\gamma - c}, \label{eq:52}
\end{gather}
where $\mathfrak M = (-X^{\gamma-c-\delta},X^{\gamma-c-\delta})$. Thus, to complete the proof of one of Theorems \ref{thm2} or \ref{thm3}, one needs only establish the minor arc bound
\begin{equation}\label{eq:60} 
\int_{\mathfrak m} S(\theta; X)^se(-N\theta) \, \dtheta \ll \tau X^{s\gamma - c - \delta}, 
\end{equation}
where $\mathfrak m = \big\{ \theta : X^{\gamma - c - \delta} \le |\theta| \le X^{\delta} \big\}$. Recalling \eqref{eq:lem10.0}, we see that \eqref{eq:60} follows from the estimate
\begin{equation}\label{eq:61} 
\sup_{\theta \in \mathfrak m} |S(\theta; X)|^{s-2} \ll X^{(s-1)\gamma - c - 2\delta}. 
\end{equation}
As this bound follows from Corollary \ref{lem:S} when $s=3$ and from Corollary \ref{lem:Sa} when $s=4$, the proofs of the two theorems are complete. 

\subsection{Comments on the proof of Theorem \ref{thm4}}

When $s=2$ and $N \sim Z$, we can set $X = (2Z/3)^{1/c}$ and $\tau = (\log Z)^{-1}$ and then structure the proof similarly to the case $s=4$, replacing the pointwise bounds \eqref{eq:52} and \eqref{eq:60} with the mean-square inequalities
\begin{equation}\label{eq:62} 
\int_{Z/2}^Z \bigg| \int_{\mathfrak M} \big( S(\theta; X)^2 - V(\theta; X)^2 \big) e(-N\theta) \, \dtheta \bigg|^2 dN \ll \tau^2X^{4\gamma - c - \eta(X)} 
\end{equation}
and 
\begin{equation}\label{eq:63} 
\int_{Z/2}^Z \bigg| \int_{\mathfrak m} S(\theta; X)^2e(-N\theta) \, \dtheta \bigg|^2 dN \ll \tau^2X^{4\gamma - c - \delta}. 
\end{equation}
From these inequalities, we see immediately that the bounds \eqref{eq:52} and \eqref{eq:60} with $s = 2$ fail for a set of Lebesgue measure $\ll Z^{1-\eta(X)}$. To complete the proof, we remark that an appeal to Plancherel's theorem (see (4.6) in \cite{KL02}) deduces \eqref{eq:62} and \eqref{eq:63} from the basic estimates for $S(\theta;X)$ (e.g., the case $s=4$ of \eqref{eq:61}) that were used earlier to establish \eqref{eq:52} and \eqref{eq:60}.

\bibliographystyle{amsplain}
\providecommand{\bysame}{\leavevmode\hbox to3em{\hrulefill}\thinspace}
\providecommand{\href}[2]{#2}

\end{document}